\documentclass[11pt,reqno]{amsart}
\usepackage[margin=3cm]{geometry}
\usepackage{graphics, amsmath,amssymb, enumitem, comment, url, mathtools}
\usepackage{graphicx}
\usepackage{epsfig}

\newif\ifcolorcomments

\usepackage{mathrsfs}  
\colorcommentstrue
\usepackage{amssymb}
\usepackage{amsmath,amsthm}

\usepackage{colortbl}

\newtheorem{theorem}{Theorem}[section]
\newtheorem{lemma}[theorem]{Lemma}
\newtheorem{proposition}[theorem]{Proposition}
\newtheorem{corollary}[theorem]{Corollary}
\theoremstyle{definition}

\newtheorem{remark}[theorem]{Remark}

\newcommand{\F}{\mathbb F}

\newcommand{\R}{\mathbb R}

\newcommand{\Z}{\mathbb Z}

\renewcommand{\text}{\textup}

\newcommand{\NPC}[1]{\ignorespaces}

\newif\ifdraft\drafttrue

\def\Z{\mathbb Z}

\def\R{\mathbb R}

\IfFileExists{marvosym.sty}{
\RequirePackage{marvosym} 
}{

}

\IfFileExists{wasysym.sty}{
\RequirePackage{wasysym}
}{

}

\mathtoolsset{showonlyrefs}
\newcommand*{\myDots}{\ifmmode\mathellipsis\else.\kern-0.07em.\kern-0.07em.\fi}

\newcommand{\cL}{\mathcal L}
\newcommand{\balpha}{\boldsymbol{\alpha}}
\newcommand{\dist}{\operatorname{dist}}
\newcommand{\card}{\#}

\usepackage{amsmath}

\newcommand {\ignore}[1] {}

\title[Nine distances and best approximations]{Nine-distance theorem and growth of best-approximation denominators}

\author[Nikita Shulga]{Nikita Shulga}
\address{Nikita Shulga, Sydney Mathematical Research Institute, The University of Sydney, NSW, Australia}
\email{nikita.shulga@sydney.edu.au}


\begin{document}

\begin{abstract}
We prove a nine-distance theorem for Kronecker sequences on flat three-tori.
That is, we show that among the first $N$ orbit points, at most nine distinct positive
nearest-neighbour distances occur. This proves the conjecture of Haynes and
Marklof. An example of Dettmann shows that nine is optimal. More generally, we prove that on a flat $d$-dimensional torus the number of such distances is at most
$2^d+1$.

The main tool is a new growth theorem for the denominators
$q_1<q_2<\cdots$ of best simultaneous approximations in a $d$-dimensional
inner-product space, which is of independent interest. We prove that, whenever $q_{n+2^d}$ is defined, either $q_{n+2^d}\ge2q_{n+1}$, or the indices $1,\ldots,2^d$ can be partitioned into disjoint pairs $\{j,k\}$, $j<k$, such that
$q_{n+k}=q_n+q_{n+j}$. In particular,
$$
 q_{n+2^d}\ge \min\{2q_{n+1},q_n+q_{n+2^{d-1}}\}\ge q_n+q_{n+1}.
$$
\end{abstract}

\maketitle

\section{Introduction}

The classical three-distance theorem (also known as the Steinhaus conjecture)
states that the gaps determined by a finite one-dimensional Kronecker
sequence have at most three distinct lengths
\cite{Swierczkowski1958,Slater1967}, see also
\cite{MarklofStrombergsson2017}. We study the higher-dimensional
nearest-neighbour analogue of this problem. Biringer and Schmidt considered
nearest-neighbour distances in the more general setting of Riemannian
manifolds \cite{BiringerSchmidt2008}.

Let $V$ be a $d$-dimensional real inner-product space with induced norm
$\|\cdot\|$, let $\cL\subset V$ be a full-rank lattice, and fix
$\balpha\in V$. We use bold $\balpha$ for the rotation vector and reserve
$\alpha$ for the one-dimensional case. The associated Kronecker sequence on
the flat torus $V/\cL$ is
$$
 \balpha+\cL,\ 2\balpha+\cL,\ 3\balpha+\cL,\ldots.
$$
For $N\ge2$ and $1\le n\le N$, set
\begin{equation}\label{eq:distance-def}
 \delta_{n,N}
 =\min\left\{\|(m-n)\balpha+\ell\|:
 1\le m\le N,\ \ell\in\cL,\ (m-n)\balpha+\ell\ne0\right\}.
\end{equation}
Thus $\delta_{n,N}$ is the least positive norm among vectors representing
displacements from $n\balpha+\cL$ to the first $N$ orbit points. As in Haynes
and Marklof \cite{HaynesMarklof2022}, the choice $m=n$ is allowed, so a
nonzero lattice vector from a torus point to itself may realise the minimum.
The final condition in \eqref{eq:distance-def} excludes zero distances caused
by repetitions in a finite orbit. Define
\begin{equation}\label{eq:g-def}
 g_N(\balpha,\cL)=\card\{\delta_{n,N}:1\le n\le N\}.
\end{equation}

The first bound of $g_N(\balpha,\cL)\leq 3^n+1$ was obtained in \cite{MR2443351}. Haynes and Marklof \cite{HaynesMarklof2022} obtained this upper bound for $g_N$ by a geometric lattice
argument and proved that the planar bound of five is sharp. Namely, they showed that five distances
occur for almost every rotation vector for infinitely many $N$. In dimension three their bound is $13$, and they
conjectured that the optimal universal bound is $9$.

Our main result gives a dimension-uniform bound and, in dimension three,
settles this conjecture.

\begin{theorem}\label{thm:main}
For every $d\ge1$, every $d$-dimensional real inner-product space $V$, every
full-rank lattice $\cL\subset V$, every $\balpha\in V$, and every $N\ge2$,
$$
 g_N(\balpha,\cL)\le2^d+1.
$$
\end{theorem}

As an immediate corollary, we obtain the promised three-dimensional
statement.

\begin{theorem}[Nine-distance theorem]\label{thm:nine-distance}
For every three-dimensional real inner-product space $V$, every full-rank
lattice $\cL\subset V$, every $\balpha\in V$, and every $N\ge2$,
$$
 g_N(\balpha,\cL)\le9.
$$
\end{theorem}

Dettmann \cite[Section~3.3]{Dettmann2025} found an example on the cubic
three-torus with
$$
 \balpha=\frac1{1334}(27,97,514),\qquad N=58,
$$
for which nine nearest-neighbour distances occur. Hence the nine-distance
theorem is sharp. Together with the classical one-dimensional result and the
sharp planar result of Haynes and Marklof, this also shows that
Theorem~\ref{thm:main} is sharp in dimensions $1$, $2$, and $3$.

\begin{remark}
Theorems~\ref{thm:main} and~\ref{thm:nine-distance} are not restricted to the
standard Euclidean norm on $\R^d$. Indeed, every $d$-dimensional real
inner-product space is linearly isometric to $\R^d$ with its standard inner
product. Under such an isometry, $\cL$ is carried to a full-rank lattice and
$\balpha$ to a rotation vector, while all distances, best-approximation
denominators, and nearest-neighbour distances are preserved. Thus the bounds
in Theorems~\ref{thm:main},~\ref{thm:nine-distance}, and~\ref{thm:growth}
hold for every norm induced by an inner product.
\end{remark}

The main tool to get Theorem~\ref{thm:main} is a new growth theorem for best
simultaneous approximation denominators. For $q\ge1$, put
$$
 \rho(q)=\dist(q\balpha,\cL).
$$
The best approximation denominators are the integers
$q_1<q_2<\cdots$ at which $\rho(q)$ attains a new strict minimum. Their growth
controls how many relevant record minima can occur in the range of
denominators that contributes to the nearest-neighbour distances.

Lagarias \cite{Lagarias1982} initiated the study of denominator growth for best
simultaneous approximations. After choosing lattice coordinates, his results
apply to the setting above. For an arbitrary norm on $\R^d$, he proved
\begin{equation}\label{eq:Lagarias-general}
 q_{n+2^{d+1}}\ge 2q_{n+1}+q_n.
\end{equation}
For the supremum norm he obtained the sharper bound
\begin{equation}\label{eq:Lagarias-sup}
 q_{n+2^d}\ge q_{n+1}+q_n.
\end{equation}
The latter holds for every $n$ when no coordinate of the vector is rational,
and for all sufficiently large $n$ for every vector outside $\mathbb Q^d$
\cite[Theorems~2.2--2.3]{Lagarias1982}. In dimension two, Lagarias used an
additional isolation argument to prove the supremum-norm growth-rate bound
determined by $x^4=x^2+1$ \cite[Theorem~1.2 and
Corollary~3.2]{Lagarias1982}.

A second classical estimate is expressed in terms of the contact number (also
called the kissing number) $K$ of the unit ball:
\begin{equation}\label{eq:contact-growth}
 q_{n+K}\ge q_{n+1}+q_n.
\end{equation}
Ermakov and Shutov attribute this general recurrence to Romanov
\cite{Romanov2006,Ermakov2010,Shutov2024}. For the Euclidean norm in dimension
two, Romanov proved the stronger four-step inequality
$$
 q_{n+4}\ge q_{n+1}+q_n.
$$
Ermakov subsequently improved the corresponding universal lower bound for the
exponential growth rate from $1.220744\ldots$ to $1.228043$
\cite{Ermakov2010}. For the two-dimensional supremum norm, Moshchevitin proved
that $q_{n+3}\ge q_{n+1}+q_n$ for infinitely many $n$
\cite{Moshchevitin2005}. Shutov gives a concise overview of these recurrences
and their applications to multidimensional distance problems
\cite{Shutov2024}.

Our denominator-growth theorem has the same index $2^d$ as Lagarias's
supremum-norm recurrence, but applies to every inner-product norm and includes
a description of the exceptional case.

\begin{theorem}\label{thm:growth}
Put $M=2^d$. Whenever $q_{n+M}$ is defined, one of the following alternatives
holds:
\begin{enumerate}
\item
$$
 q_{n+M}\ge2q_{n+1};
$$
\item the set $\{1,\ldots,M\}$ can be partitioned into disjoint pairs $\{j,k\}$,
$j<k$, such that
\begin{equation}\label{eq:pairing-recurrence}
 q_{n+k}=q_n+q_{n+j}.
\end{equation}
\end{enumerate}
Consequently, we always have
\begin{equation}\label{eq:refined-growth}
 q_{n+M}\ge
 \min\{2q_{n+1},\,q_n+q_{n+M/2}\}
 \ge q_n+q_{n+1}>2q_n.
\end{equation}
\end{theorem}

If the sequence of best denominators is infinite, Theorem~\ref{thm:growth}
also implies
$$
 \liminf_{n\to\infty}q_n^{1/n}\ge
 \begin{cases}
  (1+\sqrt5)/2,&d=1,\\
  2^{1/(2^d-1)},&d\ge2.
 \end{cases}
$$

Chevallier established a direct relation between nearest-neighbour distances
and best approximation denominators
\cite{ChevallierDistances1996,Chevallier1996}. Combined with Romanov's
four-step Euclidean recurrence, Chevallier's lemma already implies the planar
upper bound of five distances. More generally, combining it with the
contact-number recurrence \eqref{eq:contact-growth} gives the bound $K+1$, and
therefore $\sigma_d+1$ for the Euclidean norm. Thus the upper-bound parts of
the later Haynes--Marklof results were already implicit in the combination of
Chevallier's and Romanov's work. Shutov made this implication explicit
\cite{Shutov2024}. For other results on denominators of best simultaneous approximation, see \cite{MR4732677} and references therein.

In Section~4 we provide a reformulation of the argument from \cite{HaynesMarklof2022} in terms of the growth of the best-approximation denominators, which was in fact proven earlier by \cite{Romanov2006}, though his approach is very different from ours. Nevertheless, applied to the distance problem, it gives
$$
 g_N(\balpha,\cL)\le \sigma_d^{>}+1,
$$
where
$$
 \sigma_d^{>}
 =\max\left\{|C|:C\subset V,\ \|u\|=1\ (u\in C),\
    \ \langle u,v\rangle<\frac12\ \text{for }u\ne v\right\}.
$$
The quantity $\sigma_d^{>}$ depends only on $d$ and satisfies
$\sigma_d^{>}\le\sigma_d$. This known angular argument is included to compare
its consequences with the new $2^d$ recurrence and to retain the strict
$60^\circ$ separation already present in the Haynes--Marklof proof.

In Section 6, we compare two obtained bounds.

\section{Best approximation denominators}

Let $V$ be a real inner-product space of dimension $d$, with norm
$\|\cdot\|$, and let $\cL\subset V$ be a full-rank lattice. Fix
$\balpha\in V$ and define
$$
 \rho(q)=\dist(q\balpha,\cL)=\min_{p\in\cL}\|q\balpha-p\|,
 \qquad q\in\Z_{\ge1}.
$$
The sequence of best approximation denominators is defined by
\begin{equation}\label{eq:best-def}
 q_1=1,\qquad
 q_{n+1}=\min\{q>q_n:\rho(q)<\rho(q_n)\},
\end{equation}
whenever the set on the right is nonempty. If $\balpha+\cL$ has infinite
order in $V/\cL$, this gives an infinite sequence. If it has finite order, the
sequence terminates at a denominator with approximation error zero.

Write
$$
 r_n=\rho(q_n).
$$
By definition, the sequence $(r_n)_n$ is strictly decreasing. We shall use the
following record property throughout.

\begin{lemma}\label{lem:record}
For every defined best denominator $q_n$,
$$
 1\le q<q_n \quad\Longrightarrow\quad \rho(q)>r_n.
$$
\end{lemma}

\begin{proof}
The case $n=1$ is trivial. Suppose the claim holds for $n-1$. If $q<q_n$,
then either $q\le q_{n-1}$, in which case
$\rho(q)\ge r_{n-1}>r_n$, or $q_{n-1}<q<q_n$, in which case the minimality
in \eqref{eq:best-def} gives $\rho(q)\ge r_{n-1}>r_n$.
\end{proof}

\section{Sum-free sets and a proof of Theorem \ref{thm:growth}}

We need the following classical sum-free bound and its equality case, see, for example, \cite{ClarkPedersen1992,GreenRuzsa2005}. We include the proof because it is short.

\begin{lemma}\label{lem:sum-free}
Let $m\ge1$ and let $A\subset \F_2^m\setminus\{0\}$. If no three distinct
elements of $A$ have sum zero, then
$$
 |A|\le 2^{m-1}.
$$
If equality holds, then $A$ is a nonzero coset of a codimension-one subspace
of $\F_2^m$.
\end{lemma}

\begin{proof}
The assertion is trivial if $A=\varnothing$, so fix $a\in A$. We claim that
$A$ and $a+A$ are disjoint. Otherwise $x=a+y$ for some $x,y\in A$. The
elements $a,x,y$ are distinct: $x=a$ would give $y=0\notin A$, $y=a$ would
give $x=0\notin A$, and $x=y$ would give $a=0$. Hence $a+x+y=0$, contrary
to the hypothesis. Since translation preserves cardinality,
$$
 2|A|=|A|+|a+A|\le |\F_2^m|=2^m.
$$

Suppose now that equality holds. Then
$$
 \F_2^m=A\sqcup(a+A).
$$
Set $H=a+A$. We have $0\in H$. If $h_1=a+x$ and $h_2=a+y$ belong to $H$,
then $h_1+h_2=x+y$. This lies in $H$: it is $0$ when $x=y$, and otherwise
it cannot lie in $A$ by the hypothesis. Thus $H$ is a subgroup of index $2$,
and $A=a+H$.
\end{proof}

We shall also use the following elementary identity, which is essentially the parallelogram law.

\begin{lemma}\label{lem:four-sign}
For all $a,b,c$ in a real inner-product space,
\begin{align*}
 &\|a+b-c\|^2+\|a-b+c\|^2
   +\|-a+b+c\|^2+\|a+b+c\|^2 \\
 &\hspace{42mm}=4\bigl(\|a\|^2+\|b\|^2+\|c\|^2\bigr).
\end{align*}
\end{lemma}

\begin{proof}
Expanding the squared norms, each mixed inner-product term occurs twice with a
positive sign and twice with a negative sign, while each of $\|a\|^2$,
$\|b\|^2$, and $\|c\|^2$ occurs four times.
\end{proof}

\begin{proof}[Proof of Theorem~\ref{thm:growth}]
Put $M=2^d$ and write
$$
 Q_i=q_{n+i}\qquad(0\le i\le M).
$$
If $Q_M\ge2Q_1$, the first alternative holds, so suppose that
\begin{equation}\label{eq:exceptional-block}
 Q_M<2Q_1.
\end{equation}
For each $i$, choose $P_i\in\cL$ such that
$$
 x_i=Q_i\balpha-P_i,
 \qquad \|x_i\|=\rho(Q_i)=:r_i,
$$
and put $R=r_0$. Thus
\begin{equation}\label{eq:errors}
 R=r_0>r_1>\cdots>r_M.
\end{equation}
Since $Q_1$ is the first denominator after $Q_0$ at which the approximation
error is strictly smaller than $R$, we have
\begin{equation}\label{eq:below-Q1}
 0<q<Q_1\quad\Longrightarrow\quad \rho(q)\ge R.
\end{equation}

Choose a basis of $\cL$, and let $p_i\in\Z^d$ be the coordinate vector of
$P_i$. Set
$$
 z_i=(Q_i,p_i)\in\Z^{d+1},
 \qquad b_i=z_i\pmod 2\in\F_2^{d+1}.
$$
We first show that
$$
 \mathcal A=\{b_1,\ldots,b_M\}
$$
is a sum-free subset of $\F_2^{d+1}$ of cardinality $M$.

\smallskip
\noindent\emph{The classes are nonzero.}
If $b_i=0$ for some $1\le i\le M$, then $z_i/2$ is integral and
$$
 0<\frac{Q_i}{2}<Q_1
$$
by \eqref{eq:exceptional-block}. The vector $x_i/2$ is admissible in the
definition of $\rho(Q_i/2)$, and therefore
$$
 \rho(Q_i/2)\le\frac{r_i}{2}<R,
$$
contrary to \eqref{eq:below-Q1}.

\smallskip
\noindent\emph{The classes are distinct.}
Suppose $b_i=b_j$ with $1\le i<j\le M$. Then $(z_j-z_i)/2$ is integral and
$$
 0<\frac{Q_j-Q_i}{2}<Q_1.
$$
Consequently,
$$
 \rho\left(\frac{Q_j-Q_i}{2}\right)
 \le\frac12\|x_j-x_i\|
 \le\frac{r_i+r_j}{2}<R,
$$
again contradicting \eqref{eq:below-Q1}.

\smallskip
\noindent\emph{No three classes sum to zero.}
Suppose, to the contrary, that
\begin{equation}\label{eq:zero-sum-triple}
 b_i+b_j+b_k=0,
 \qquad 1\le i<j<k\le M.
\end{equation}
Define
$$
 \begin{aligned}
 A&=x_i+x_j-x_k,&
 B&=x_i-x_j+x_k,\\
 C&=-x_i+x_j+x_k,&
 D&=x_i+x_j+x_k.
 \end{aligned}
$$
The four analogous signed combinations of $z_i,z_j,z_k$ lie in
$2\Z^{d+1}$. Moreover,
$$
 0<\frac{Q_i+Q_j-Q_k}{2}<Q_1,
 \qquad
 0<\frac{Q_i-Q_j+Q_k}{2}<Q_1.
$$
Indeed, the lower bound in the first inequality follows from
$Q_i\ge Q_1$, $Q_j>Q_1$, and $Q_k<2Q_1$; all the other bounds follow from
the ordering of the $Q_i$ and \eqref{eq:exceptional-block}. Hence
\eqref{eq:below-Q1} gives
\begin{equation}\label{eq:AB-long}
 \|A\|\ge2R,\qquad \|B\|\ge2R.
\end{equation}

Lemma~\ref{lem:four-sign}, applied to $x_i,x_j,x_k$, gives
$$
 \|A\|^2+\|B\|^2+\|C\|^2+\|D\|^2
 =4\bigl(r_i^2+r_j^2+r_k^2\bigr).
$$
Using \eqref{eq:AB-long} and $\|D\|^2\ge0$, we obtain
$$
 \begin{aligned}
 \|C\|^2
 &\le4\bigl(r_i^2+r_j^2+r_k^2\bigr)-8R^2\\
 &=4r_k^2+4\bigl(r_i^2+r_j^2-2R^2\bigr)
 <4r_k^2,
 \end{aligned}
$$
because $i,j\ge1$ and hence $r_i,r_j<R$. Thus
\begin{equation}\label{eq:C-short}
 \|C\|<2r_k.
\end{equation}
On the other hand, $(-z_i+z_j+z_k)/2$ is integral and has denominator
$$
 t=\frac{-Q_i+Q_j+Q_k}{2},
 \qquad 0<t<Q_k.
$$
Therefore
$$
 \rho(t)\le\frac{\|C\|}{2}<r_k=\rho(Q_k),
$$
contrary to Lemma~\ref{lem:record}. This proves that $\mathcal A$ is
sum-free.

Since $|\mathcal A|=M=2^d=\frac12|\F_2^{d+1}|$,
Lemma~\ref{lem:sum-free} gives
\begin{equation}\label{eq:affine-hyperplane}
 \mathcal A=a+H,
\end{equation}
where $H$ is a codimension-one subspace and $a\notin H$.

We next show that $b_0$ is a nonzero element of $H$. If $b_0=0$, then
$Q_0/2<Q_1$ and
$$
 \rho(Q_0/2)\le R/2<R,
$$
contrary to \eqref{eq:below-Q1}. If $b_0=b_i$ for some $1\le i\le M$, then
$$
 0<\frac{Q_i-Q_0}{2}<Q_1
$$
and
$$
 \rho\left(\frac{Q_i-Q_0}{2}\right)
 \le\frac{r_i+R}{2}<R,
$$
again a contradiction. Thus $b_0\notin\mathcal A$ and $b_0\ne0$, so
\eqref{eq:affine-hyperplane} implies $b_0\in H$.

Translation by $b_0$ is therefore a fixed-point-free involution of
$\mathcal A$. It partitions $\{1,\ldots,M\}$ into pairs $\{j,k\}$ such that,
after ordering each pair with $j<k$,
\begin{equation}\label{eq:parity-pair}
 b_k=b_0+b_j.
\end{equation}
We claim that every such pair satisfies $Q_k=Q_0+Q_j$. Suppose otherwise.
The relation \eqref{eq:parity-pair} implies that the signed combinations of
$z_0,z_j,z_k$ used below are even. The two nonzero denominators
$$
 \frac{|Q_0+Q_j-Q_k|}{2},
 \qquad
 \frac{Q_0-Q_j+Q_k}{2}
$$
both lie in $(0,Q_1)$, by the ordering of the $Q_i$ and
\eqref{eq:exceptional-block}. The corresponding vectors are,
up to sign,
$$
 x_0+x_j-x_k,\qquad x_0-x_j+x_k.
$$
Hence \eqref{eq:below-Q1} shows that both of these vectors have norm at
least $2R$. The four-sign identity, now applied
to $x_0,x_j,x_k$, then gives
$$
 \|-x_0+x_j+x_k\|^2
 \le4(R^2+r_j^2+r_k^2)-8R^2
 <4r_k^2.
$$
But $(-z_0+z_j+z_k)/2$ has a positive denominator smaller than $Q_k$, which
again contradicts Lemma~\ref{lem:record}. Hence
$$
 Q_k=Q_0+Q_j,
$$
proving \eqref{eq:pairing-recurrence}.

Finally, let $\{s,M\}$ be the pair containing $M$. There are $M/2$ smaller
members of the pairs, and $s$ is the largest of them, because the map
$Q_j\mapsto Q_0+Q_j$ preserves order. Hence $s\ge M/2$ and
$$
 Q_M=Q_0+Q_s\ge Q_0+Q_{M/2}.
$$
This proves \eqref{eq:refined-growth}; its final strict inequality follows
from $Q_1>Q_0$ and $Q_{M/2}>Q_0$.
\end{proof}

\begin{corollary}\label{cor:growth-rate}
Suppose that the sequence of best denominators is infinite. Then
$$
 \liminf_{n\to\infty}q_n^{1/n}\ge
 \begin{cases}
  \dfrac{1+\sqrt5}{2},&d=1,\\
  2^{1/(2^d-1)},&d\ge2.
 \end{cases}
$$
\end{corollary}

\begin{proof}
For $d=1$, Theorem~\ref{thm:growth} gives
$$
q_{n+2}\ge q_{n+1}+q_n.
$$
Put
$$
\varphi=\frac{1+\sqrt5}{2},
$$
so that $\varphi^2=\varphi+1$. Choose $c>0$ such that
$$
q_1\ge c\varphi,\qquad q_2\ge c\varphi^2.
$$
Suppose inductively that
$$
q_n\ge c\varphi^n,\qquad q_{n+1}\ge c\varphi^{n+1}.
$$
Then
$$
q_{n+2}\ge q_{n+1}+q_n\ge c\varphi^n(\varphi+1)=c\varphi^{n+2}.
$$
Thus $q_n\ge c\varphi^n$ for every $n$, and consequently
$$
\liminf_{n\to\infty}q_n^{1/n}\ge\varphi.
$$
Now suppose that $d\ge2$. Put $M=2^d$ and set $\lambda=2^{1/(M-1)}.$
Choose $c>0$ such that $q_j\ge c\lambda^j$ for all $1\le j\le M.$
We prove by induction that $q_j\ge c\lambda^j$
for every $j\ge1$. Suppose that the claim holds for all indices less than
$n+M$. If the first alternative of Theorem~\ref{thm:growth} holds, then
$$
q_{n+M} \ge 2q_{n+1}\ge 2c\lambda^{n+1}=c\lambda^{n+M},
$$
because $\lambda^{M-1}=2$.

In the second alternative, there exists $s\ge M/2$ such that
$$
 q_{n+M}=q_n+q_{n+s}.
$$
The induction hypothesis therefore gives
$$
 q_{n+M} \ge c\lambda^n+c\lambda^{n+s} \ge c\lambda^n\bigl(1+\lambda^{M/2}\bigr).
$$
It remains to verify that $1+\lambda^{M/2}\ge\lambda^M.$
Put $y=\lambda^{M/2}$. Since $M\ge4$,
$$
 y=2^{M/(2(M-1))}
 \le 2^{2/3}
 <\frac{1+\sqrt5}{2}.
$$
Since the positive root of $x^2=x+1$ is $(1+\sqrt5)/2$, it follows that $y^2<y+1.$
Hence
$$
 \lambda^M=y^2<1+y=1+\lambda^{M/2},
$$
and consequently
$$
 q_{n+M}\ge c\lambda^{n+M}.
$$
This completes the induction. Therefore
$$
 \liminf_{n\to\infty}q_n^{1/n}
 \ge\lambda
 =2^{1/(2^d-1)}.
$$

%
\end{proof}

In dimension $2$, Corollary~\ref{cor:growth-rate} gives
$$
 \liminf_{n\to\infty}q_n^{1/n}\ge2^{1/3}=1.259921\ldots,
$$
improving the lower bound $1.228043$ obtained by Ermakov
\cite{Ermakov2010}. The proof above uses the inner-product identity
of Lemma~\ref{lem:four-sign}. It does not directly extend to a general norm.

\section{Bound on growth with kissing numbers}

The following is essentially the result of Romanov \cite{Romanov2006} and Haynes and Marklof \cite[Propositions~18--19]{HaynesMarklof2022}, which was stated in different terms. Recall that $\sigma_d^{>}$ is the largest
number of unit vectors in $V$ with
pairwise inner products strictly less than $1/2$. Equivalently, their pairwise
angles are strictly greater than $60^\circ$. Since the ordinary kissing number
$\sigma_d$ permits inner product $1/2$, we have
$\sigma_d^{>}\le\sigma_d$. As our proof is different from \cite{Romanov2006} and provides a minor refinement by not considering the inner product equal to $1/2$, we provide the proof in full.

\begin{theorem}\label{thm:strict-growth}
With the notation of Section~2,
$$
 q_{n+\sigma_d^{>}}\ge q_n+q_{n+1}
$$
whenever $q_{n+\sigma_d^{>}}$ is defined.
\end{theorem}

\begin{proof}
Put $M=\sigma_d^{>}$ and suppose, to the contrary, that
$$
 Q_0<Q_1<\cdots<Q_M<Q_0+Q_1,
 \qquad Q_i=q_{n+i}.
$$
Choose nearest representatives
$$
 x_i=Q_i\balpha-P_i,
 \qquad r_i=\|x_i\|,
 \qquad R=r_0.
$$
As in \eqref{eq:below-Q1},
\begin{equation}\label{eq:angular-threshold}
 0<q<Q_1\quad\Longrightarrow\quad \rho(q)\ge R.
\end{equation}
If $r_M=0$, then
$$
 0<Q_M-Q_1<Q_0,\qquad
 \rho(Q_M-Q_1)=\rho(Q_1)=r_1<R,
$$
contrary to Lemma~\ref{lem:record}. Hence all the vectors $x_i$ are nonzero.

For $i<j$, put $q=Q_j-Q_i$. Then
$$
 0<q\le Q_M-Q_0<Q_1,
$$
and, since $P_j-P_i\in\cL$,
$$
 x_j-x_i=q\balpha-(P_j-P_i).
$$
Consequently,
\begin{equation}\label{eq:strict-pair}
 \|x_i-x_j\|\ge\rho(q)\ge R.
\end{equation}
Put $u_i=x_i/r_i$, so that $\|u_i\|=1$. If $i<j$ and
$\langle u_i,u_j\rangle\ge1/2$, then $r_i>r_j>0$ and
$$
 \begin{aligned}
 \|x_i-x_j\|^2
 &\le r_i^2+r_j^2-r_i r_j\\
 &=r_i^2-r_j(r_i-r_j)
 <r_i^2\le R^2,
 \end{aligned}
$$
contrary to \eqref{eq:strict-pair}. Thus
$$
 \langle u_i,u_j\rangle<\frac12
 \qquad(i\ne j),
$$
so $u_0,\ldots,u_M$ are $\sigma_d^{>}+1$ unit vectors with pairwise inner
products strictly less than $1/2$. This contradicts the definition of
$\sigma_d^{>}$.
\end{proof}

\section{Nearest-neighbour distances}

The following proposition expresses the quantities in \eqref{eq:distance-def}
in terms of the best approximation denominators. Let
$$
 \lambda_1(\cL)=\min\{\|\ell\|:\ell\in\cL\setminus\{0\}\}
$$
be the length of a shortest nonzero lattice vector, and put
$$
 \psi(m)=\min_{1\le q\le m}\rho(q),
 \qquad m\ge1.
$$

\begin{proposition}\label{prop:distance-formula}
If the cosets $n\balpha+\cL$, $1\le n\le N$, are distinct, then
\begin{equation}\label{eq:delta-formula}
 \delta_{i,N}
 =\min\left\{\lambda_1(\cL),
 \psi\bigl(\max\{i-1,N-i\}\bigr)\right\}.
\end{equation}
Consequently,
\begin{equation}\label{eq:g-formula}
 g_N(\balpha,\cL)
 =1+\card\left\{j:
 \left\lfloor\frac N2\right\rfloor<q_j\le N-1,
 \quad r_j<\lambda_1(\cL)\right\}.
\end{equation}
In particular,
\begin{equation}\label{eq:g-formula-upper}
 g_N(\balpha,\cL)
 \le 1+\card\left\{j:
 \left\lfloor\frac N2\right\rfloor<q_j\le N-1\right\}.
\end{equation}
\end{proposition}

\begin{proof}
Write $k=m-i$ in \eqref{eq:distance-def}. The choice $k=0$ contributes
precisely the shortest nonzero lattice-vector length $\lambda_1(\cL)$. For
$k\ne0$, the possible absolute index differences form the set
$$
 \{1,2,\ldots,\max\{i-1,N-i\}\},
$$
and
$$
 \min_{\ell\in\cL}\|k\balpha+\ell\|=\rho(|k|).
$$
This proves \eqref{eq:delta-formula}.

As $i$ ranges from $1$ to $N$, the integer $\max\{i-1,N-i\}$ assumes
every value from $\lfloor N/2\rfloor$ through $N-1$. The function
$m\mapsto\min\{\lambda_1(\cL),\psi(m)\}$ changes value at a best denominator
$q_j$ precisely when $r_j<\lambda_1(\cL)$. This proves \eqref{eq:g-formula} and the inequality \eqref{eq:g-formula-upper} follows.
\end{proof}

\begin{proof}[Proof of Theorem~\ref{thm:main}]
First suppose that the $N$ points are distinct. By
\eqref{eq:g-formula-upper}, it suffices to count the best denominators in
$\bigl(\lfloor N/2\rfloor,N-1\bigr]$. If that interval contained $2^d+1$
best denominators, then for some $n$,
$$
 \left\lfloor\frac N2\right\rfloor<q_n<q_{n+1}<\cdots<q_{n+2^d}\le N-1.
$$
But
$$
 N-1<2\left(\left\lfloor\frac N2\right\rfloor+1\right)\le2q_n,
$$
so $q_{n+2^d}<2q_n$, contradicting Theorem~\ref{thm:growth}. Hence this
interval contains at most $2^d$ best denominators, and
\eqref{eq:g-formula-upper} gives $g_N\le 2^d+1$.

It remains to consider a finite-order rotation. Let $T$ be the order of
$\balpha+\cL$ in $V/\cL$. If $N<T$, the cosets are distinct and the preceding
argument applies. If $N\ge T$, the first $N$ terms contain every element of
the cyclic subgroup generated by $\balpha+\cL$. For each $n$, the nonzero
vectors $(m-n)\balpha+\ell$ occurring in \eqref{eq:distance-def} therefore
form the same set, so all the values $\delta_{n,N}$ are equal. Hence $g_N=1$,
including the case $T=1$, when this common value is $\lambda_1(\cL)$.
\end{proof}

\begin{corollary}
The optimal universal bounds in dimensions $1$, $2$, and $3$ are, respectively,
$$
 3,\qquad 5,\qquad 9.
$$
\end{corollary}

\begin{proof}
Since $2^1+1=3$, Theorem~\ref{thm:main} gives
$g_N(\alpha,\Z)\le3$. For $\alpha=3/11$ and $N=5$, the orbit points, in
increasing order, are
$$
 \frac1{11},\frac3{11},\frac4{11},\frac6{11},\frac9{11}.
$$
Their positive nearest-neighbour distances are therefore $1/11$, $2/11$, and $3/11$.
This example also appears in \cite[Section~3.1]{Dettmann2025}.

For $d=2$, Haynes and Marklof \cite{HaynesMarklof2022} proved that on the standard two-torus five distances occur for almost every $\balpha$ for infinitely many $N$. For $d=3$, Dettmann proved that on $\R^3/\Z^3$,
$$
 \balpha=\frac1{1334}(27,97,514),\qquad N=58,
$$
gives $g_N(\balpha,\Z^3)=9$ \cite[Section~3.3]{Dettmann2025}.
The corresponding upper bounds $5$ and $9$ follow from
Theorem~\ref{thm:main}.
\end{proof}

\section{Comparison of two bounds}

Theorem~\ref{thm:main} gives the bound $g_N\le 2^d+1$. On the other hand,
Theorem~\ref{thm:strict-growth}, Proposition~\ref{prop:distance-formula}, and
the same counting argument used in the proof of Theorem~\ref{thm:main} give
\begin{equation}\label{eq:kissing-distance-bound}
 g_N(\balpha,\cL)\le \sigma_d^{>}+1\le\sigma_d+1.
\end{equation}
This is the kissing-number bound mentioned in the introduction. Writing out two bounds in a table, we get
$$
\begin{array}{c|cccccccc}
d & 3 & 4 & 5 & 6 & 7 & 8 & 9 & 10 \\ \hline
\text{bound }2^d+1
  & 9 & 17 & 33 & 65 & 129 & 257 & 513 & 1025 \\
\text{bound }\sigma_d^{>}+1
  & 13 & 25 & 45 & 78 & 135 & 240 & 364 & 554
\end{array}
$$

The improvement in dimension $8$ in the kissing-number row uses the strict inequality in the definition of $\sigma_d^{>}$. Indeed, up to orthogonal transformations there is a unique set of $240$ unit
vectors in $\R^8$ with pairwise inner products at most $1/2$
\cite{BannaiSloane1981}. This is the normalised minimal-vector configuration
of the $E_8$ lattice, and some pairs have inner product exactly $1/2$
\cite{ConwaySloane1999}. Hence $\sigma_8^{>}\le239$, and
Equation~\eqref{eq:kissing-distance-bound} gives $g_N\le240$. The $196560$-point Leech kissing
configuration is likewise unique and
contains pairs with inner product exactly $1/2$
\cite{BannaiSloane1981,ConwaySloane1999}. Hence
$\sigma_{24}^{>}\le196559$, and equation~\eqref{eq:kissing-distance-bound} gives
$g_N\le196560$.

Finally, we note the Kabatiansky--Levenshtein bound implies
$$
 g_N(\balpha,\cL)
 \le 1+\sigma_d^{>}
 \le 1+2^{0.401d(1+o(1))}
 \quad d\to\infty,
$$
see \cite{KabatianskyLevenshtein1978,ConwaySloane1999}.

\bibliographystyle{amsplain}
\bibliography{growth_denominators_nine_distance}

\end{document}